\documentclass{article}

\usepackage{stmaryrd}
\usepackage{lmodern}
\usepackage{diagbox}
\usepackage{mathtools}
\usepackage{amssymb,amsmath}
\usepackage{graphicx}

\usepackage{mathrsfs}

\usepackage{color}
\usepackage{amsfonts}
\usepackage{epsfig}
\usepackage{graphics}

\usepackage{multicol}
\usepackage{lscape}
\usepackage{latexsym}
\def\proof{\noindent {\bf Proof }}
\def\Fq{{\mathbb F}_q}

\def\Zq-1{{\mathbb Z}_{q-1}}

\def\alp{{\alpha}}

\def\qed{~~\vrule height8pt width4pt depth0pt}

\def\Mcal{{\cal M}}
\def\Ical{{\cal I}}

\def\F2{{\mathbb F}_2}

\def\eps{{\varepsilon}}

\def\Mcal{{\cal M}}
\def\Ecal{{\cal E}}

\newtheorem{prop}{Proposition}

\newtheorem{lemma}{Lemma}
\newtheorem{thm}{Theorem}
\newtheorem{cor}{Corollary}
\newtheorem{example}{Example}

\newcommand{\be}{\begin{equation}}
\newcommand{\ee}{\end{equation}}
\newcommand{\beq}{\begin{eqnarray}}
\newcommand{\eeq}{\end{eqnarray}}
\newcommand{\beqn}{\begin{eqnarray*}}
\newcommand{\eeqn}{\end{eqnarray*}}

\newcommand{\fq}{{\mathbb F}_{q}}

\newcommand{\ftwo}{{\mathbb F}_{2}}

\newcommand{\Jcal}{{\cal J}}
\newcommand{\Rcal}{{\cal R}}
\newcommand{\Scal}{{\cal S}}

\newcommand{\om}{{\omega}}

\begin{document}

\title {Counting irreducible polynomials with prescribed coefficients over a finite field }
 \author{ Zhicheng Gao, Simon Kuttner, and Qiang Wang \footnote{Z. Gao's research is partially supported by NSERC of Canada(RGPIN 04010-2015) and Carleton University Development Grant (189035).  S. Kuttner is partially supported by an NSERC USRA.  Q. Wang's research is partially supported by  NSERC of Canada (RGPIN-2017-06410).}\\
School of Mathematics and Statistics\\
Carleton University\\
Ottawa, Ontario\\
Canada K1S5B6 \\
Email: \{zgao, wang\}@math.carleton.ca }
\maketitle

\begin{abstract}
We continue our study on counting irreducible polynomials over a finite field with prescribed coefficients. We set up a general combinatorial framework
using generating functions with coefficients from a group algebra which is generated by equivalent classes of polynomials with prescribed coefficients.  Simplified expressions  are derived for some special cases.
Our results extend some earlier results.
\end{abstract}

\section{ Introduction}

Let $q$ be a prime power and  $\mathbb{F}_q$ be a finite field of $q$ elements. The problem of
estimating the number of irreducible polynomials
of degree $d$ over the finite field ${\mathbb F}_{q}$
with some prescribed coefficients has been largely
studied; see surveys by S. D. Cohen such as  \cite{Cohen1} and Section 3.5  in \cite{Handbook} and references therein for more details.  Asymptotic results were answered in the most generality by Cohen \cite{Cohen0}.
Regarding to the exact  formulae or expression,  Carlitz~\cite{Car52} and Kuz'min~\cite{Kuz90}
gave the number of  monic irreducible polynomials with the
first coefficient prescribed and the first two
coefficients prescribed, respectively; see~\cite{CMRSS03, RMS01}
for a similar result over $\ftwo$, and~\cite{MoiRan08, Ri14}
for more general results.  Yucas and Mullen~\cite{YucMul04}
and Fitzgerald and Yucas~\cite{fityuc} considered  the
number of irreducible polynomials of degree $d$ over
$\ftwo$ with the first three coefficients prescribed.
Over any finite field $\fq$, Yucas~\cite{Yuc06} gave
the number of irreducible polynomials with prescribed
first or last coefficient. In \cite{YucMul04} Yucas and Mullen studied the number of irreducible polynomials over ${\mathbb  F}_2$ with the first three coefficients, and they stated: ``It would be interesting to know whether the methods and techniques of \cite{Kuz90} could be extended and used to generalize both our formulas and those of [4] to formulas for
arbitrary finite fields, and/or to the case over $\mathbb{F}_2$ where more than three coefficients
are specified in advance." Recently,   Lal\'{i}n and Larocque  \cite{LalLar16} used elementary combinatorial methods, together with the theory of quadratic forms, over finite fields
to obtain the formula, originally due to Kuz'min \cite{Kuz90}, for the
number of monic irreducible polynomials of degree $n$ over a
finite field $\fq$ with the first two prescribed coefficients.  Also, an explicit expression for  the number of irreducible polynomials over $\mathbb{F}_{2^r}$ with the first three coefficients prescribed zero was given by Ahmadi et al  \cite{Ahmadi16}; the proofs involve counting the number of points on certain algebraic curves over finite fields which are supersingular. More recently,  Granger \cite{Gra19} carried out a systematic study on the problem with several prescribed leading coefficients.
Through a transformation of the problem of counting the number of elements of $\mathbb{F}_{q^n}$ with prescribed traces into the problem of counting the number of elements for which linear combinations of the trace functions evaluate to $1$,  he converted the problem into counting points in Artin-Schreier curves of smaller genus and then computed the corresponding zeta functions using Lauder-Wan algorithm \cite{LW}.   In particular, he presented an efficient deterministic algorithm which outputs
exact expressions in terms of the degree $n$ for the number of monic degree
n irreducible polynomials over $\fq$ of characteristic $p$ for which
the first  $\ell < p$ coefficients are prescribed, provided that $n$ is coprime to $p$.

 In this paper we use the generating function approach, which is initiated in \cite{GKW21}, to study the problem with several prescribed leading and/or ending coefficients.  We study the group of equivalent classes for these polynomials with prescribed coefficients  and extend ideas from  Hayes  \cite{Hay65} and Kuz'min \cite{Kuz90}.  We also note that a similar idea was used by Fomenko   \cite{Fom96} to study the $L$ functions  for the number of  irreducible polynomials over $\mathbb{F}_2$ with prescribed three coefficients, and for the case such as  prescribed $\ell$ coefficients with $\ell < p$.  Using primitive idempotent decomposition for finite abelian group algebras,  we can obtain  general expressions for the generating functions over group algebras.  This provides us a recipe to obtain an explicit formulae for the number of monic irreducible polynomials  with prescribed leading coefficients, as well as prescribed ending coefficients.  We demonstrate our method by computing these numbers for several concrete examples.   Our method is also computationally simpler than that of Granger \cite{Gra19} in the case of prescribed leading coefficients only and it produces simpler formulas in some cases.

The rest of the paper is organized as follows. In Section~\ref{main} we described our generating function method and derive our main results. In Section~\ref{TypeI} we apply our main theorem to obtain new compact expressions for some examples with prescribed leading and ending coefficients. In Section~\ref{TypeII} we apply our main theorem to obtain compact expressions for some examples with prescribed leading coefficients and compare them with previous known results. The conclusion is in Section~\ref{conclusion}.

\section{Combinatorial framework for counting irreducible polynomials with prescribed coefficients} \label{main}

In this section, we describe our general combinatorial framework for counting irreducible polynomials with prescribed coefficients,
using generating functions with coefficients from a group algebra. This extends Kuz'Min and Hayes's  idea \cite{Kuz90, Hay65}
for polynomials with two prescribed leading coefficients.
Fix positive integers $\ell$ and $t$.
Given a polynomial  $f=x^m+f_1x^{m-1}+\cdots +f_{m-1}x+f_m$, we shall call $f_1,\ldots, f_{\ell}$ the first $\ell$ leading coefficients, and $f_m,\ldots,f_{m-t+1}$ the ending coefficients.
When we read the leading coefficients from left to right, missing coefficients are interpret as zero.
Similarly, we read the ending coefficients from right to left, and interpret the missing coefficients as 0. Thus the leading and ending coefficients of $f$
are the same as those of
\[
\sum_{j=0}^{\ell}f_jx^{\ell+t-j}+\sum_{j=0}^{t-1}f_{m-j}x^j,
\]
where $f_0:=1$, and $f_j:=0$ if $j<0$ or $j>m$.

We shall treat the following two different types. Let $\Mcal$ denote the set of monic polynomials  over $\Fq$,
$\Mcal_d$ consisting of those monic polynomials of degree $d$, and let $\deg(f)$ denote the degree of a polynomial $f$.

\medskip

\noindent Type~I.  We wish to prescribe $\ell $ leading coefficients $a_1,\ldots,a_{\ell}$, and $t$ ending coefficients $b_0,\ldots,b_{t-1}$ with the constant term $b_0\ne 0$.
 Two monic polynomials $f,g\in \Mcal$, with $f(0)g(0)\ne 0$, are said to be equivalent if they have the same $\ell$ leading and $t$ ending coefficients.
 Thus the polynomial $f=1$ is equivalent to $x^{\ell+t}+1$, and each equivalence class is represented by a unique monic polynomial of degree $\ell+t$.  We recall the reciprocal of $f$ is the polynomial of the form $x^{\deg(f)} f(x^{-1})$ and denote the reciprocals of $f$  and $g$ by $\tilde{f}$ and $\tilde{g}$ respectively.
The type I equivalence relation $f\sim g$  can be written as
\begin{align}
\tilde{f} \equiv \tilde{g} \pmod{x^{\ell+1}},  \mbox{and }
f\equiv g \pmod{x^{t}}.  \label{eq:II}
\end{align}
We emphasize that $\tilde{f}(0) = \tilde{g}(0)= 1$ and $f(0)g(0)\ne 0$.

\medskip

\noindent Type~II.   We wish to prescribe $\ell $ leading coefficients. Two monic polynomials $f,g\in \Mcal$ are said to be equivalent if they have the same $\ell$ leading coefficients.  We may write the type II equivalence relation $f\sim g$ as
\[
\tilde{f} \equiv \tilde{g} \pmod{x^{\ell+1}}.
\]


 In this case the polynomial $f=1$ is equivalent to $x^{d}$ for any $d>0$, and each  equivalence class is represented by a unique monic polynomial of degree $\ell$.


It is not difficult to see that this multiplication is well defined (independent of the choice of representatives) because the leading and ending coefficients of $fg$ are determined by the leading and ending coefficients of $f$ and $g$. 
We note that  the set  $\Ecal$ of  type~II equivalence classes is a group under the usual multiplication; see, e.g., \cite{Fom96}.  

\begin{prop}\label{group}
For both types,  $\Ecal$  is an abelian group under the multiplication $\langle f\rangle \langle g\rangle =\langle fg\rangle$ with $\langle 1\rangle$ being the identity element. Moreover, the following holds.
\begin{itemize}
\item[(I)] For  type~I, we have $|\Ecal|=(q-1)q^{\ell+t-1}$. Also, for each $d\ge \ell+t$ and each $\eps\in \Ecal$, there are exactly $q^{d-\ell-t}$ polynomials in $\Mcal_d$ which are equivalent to $\eps$.
\item[(II)] For  type~II, we have $|\Ecal|=q^{\ell}$. Also, for each $d\ge \ell$ and each $\eps\in \Ecal$, there are exactly $q^{d-\ell}$ polynomials in $\Mcal_d$ which are equivalent to $\eps$.
\end{itemize}
\end{prop}
\proof
The multiplication is well defined because the product is independent of the choices of representatives in the equivalence classes.
For type~I, we only need to show that for each monic polynomial $f$ of degree $\ell+t$, there exists a unique monic polynomial $g$ of degree $\ell+t$ such that $\langle fg\rangle=\langle 1\rangle$. Writing
\begin{align*}
f& = \sum_{j=0}^{\ell+t}f_jx^{\ell+t-j},\\
g&=\sum_{j=0}^{\ell+t}g_jx^{\ell+t-j},
\end{align*}
where $f_0=g_0=1$ and $f_{\ell+t}g_{\ell+t}\ne 0$. We have
\begin{align}\label{eq:multiply}
fg=x^{2\ell+2t}+\sum_{d=1}^{2\ell+2t-1}\sum_{j=0}^df_jg_{d-j}x^{2\ell+2t-d} +f_{\ell+t}g_{\ell+t}.
\end{align}
 Thus
$\langle fg\rangle=\langle 1\rangle$ iff
\begin{align}
\sum_{j=0}^df_jg_{d-j}&=0,~~1\le d\le \ell, \label{eq:start}\\
\sum_{j=0}^{d}f_{\ell+t-j}g_{\ell+t-d+j}&=0,~~1\le d\le t-1, \label{eq:end}\\
g_{\ell+t}&=1/f_{\ell+t}.
\end{align}
The above system uniquely determines the values of $g_1,\ldots,g_{\ell}$ by the recursion
\begin{align}
g_d&=-\sum_{j=1}^d f_jg_{d-j},~~1\le d\le \ell,\\
g_{\ell+t-d}&=\frac{-1}{f_{\ell+t}}\sum_{j=1}^{d}f_{\ell+t-j}g_{\ell+t-d+j},~~1\le d\le t-1.
\end{align}
This completes the proof for type~I.\\

For type~II, we note that $\langle x^k\rangle=\langle x^{\ell}\rangle=\langle 1\rangle$ for all $k\ge 0$. Thus only equation~\eqref{eq:start} is used for obtaining the unique inverse.
\qed

We shall use $0$ to denote the zero element of the group algebra $\mathbb{C}\Ecal$ generated by the group $\Ecal$ over the complex field $\mathbb{C}$. For the type~I equivalence, it is convenient to define $\langle f\rangle:= 0$ if $f(0)=0$.


Define the following generating function
\begin{align}
F(z)=\sum_{f\in \Mcal} \langle f\rangle z^{\deg(f)}.
\end{align}
 We note that $F(z)$ is a formal power series with coefficients in the group algebra $\mathbb{C}\Ecal$.

Let $\Ical_d$ be the set of  irreducible polynomials in $\Mcal_d$ and for each $\eps \in \Ecal$ we  define
\[
I_d(\eps)=\#\{f\in \Ical_d: \langle f\rangle=\eps\}.
\]
Since a monic polynomial is uniquely  factored into a multiset of monic irreducible polynomials. The standard counting argument (generating function argument) (see \cite{FlaSed09}, for example) leads to
\begin{align*}
F(z)=\prod_{d\ge 1}\prod_{f\in \Ical_d}\left(\langle 1\rangle-z^d\langle f\rangle\right)^{-1}
=\prod_{d\ge 1}\prod_{\eps\in \Ecal}\left(\langle 1\rangle-z^d\eps\right)^{-I_d(\eps)}.
\end{align*}
Consequently
\begin{align}
\ln F(z) &=\sum_{d\ge 1}\sum_{f\in \Ical_d}\sum_{k\ge 1}\frac{1}{k}z^{kd}\langle f\rangle^k \label{eq:Irre}  \\
&=\sum_{d\ge 1}\sum_{\eps\in \Ecal}I_d(\eps)\ln\left(\langle 1\rangle- z^d\eps\right)^{-1}\nonumber\\
&=\sum_{d\ge 1}\sum_{k\ge 1}\sum_{\eps\in \Ecal}\frac{1}{k}I_d(\eps)z^{kd}\eps^k\nonumber\\
&=\sum_{m\ge 1}\sum_{k|m}\sum_{\eps\in \Ecal}\frac{1}{k}I_{m/k}(\eps)z^m\eps^k. \label{eq:Ifactor}
\end{align}

For each $\eps\in \Ecal$, define
\begin{align}\label{eq:N1}
N_d(\eps)&=d\left[z^d\eps\right]\ln F(z).
\end{align}

\begin{prop} \label{prop:IN}
\begin{align}
I_d(\eps)=\frac{1}{d}\sum_{k|d}\mu(k)\sum_{\eps_1\in \Ecal}
N_{d/k}(\eps_1)\left\llbracket \eps_1^k=\eps \right\rrbracket.
\end{align}
\end{prop}
\proof Extracting the coefficient of $z^m\eps$ on both sides of (\ref{eq:Ifactor}), we obtain
\begin{align} \label{NofI}
N_m(\eps')&=m\sum_{k|m}\sum_{\eps\in \Ecal}\frac{1}{k}I_{m/k}(\eps)\left\llbracket \eps^k=\eps' \right\rrbracket.
\end{align}
Thus
\begin{align*}
& \frac{1}{d}\sum_{k|d}\mu(k)\sum_{\eps_1}
N_{d/k}(\eps_1)\left\llbracket \eps_1^k=\eps \right\rrbracket\\
&=\frac{1}{d}\sum_{k|d}\mu(k)\sum_{\eps_1}\frac{d}{k}\sum_{j|(d/k)}\sum_{\eps_2}\frac{1}{j}I_{d/kj}(\eps_2)\left\llbracket \eps_2^j=\eps_1 \right\rrbracket\left\llbracket \eps_1^k=\eps \right\rrbracket\\
&=\sum_{t|d}\left(\sum_{k|t}\mu(k)\right)\frac{1}{t}\sum_{\eps_2}I_{d/t}(\eps_2)\left\llbracket \eps_2^t=\eps \right\rrbracket ~~\hbox{(set $t=jk$)}\\
&=\sum_{t|d}\llbracket t=1\rrbracket \frac{1}{t}\sum_{\eps_2}I_{d/t}(\eps_2)\left\llbracket \eps_2^t=\eps \right\rrbracket\\
&=I_d(\eps). \qed
\end{align*}


The following result from \cite[Proposition~3.1]{EM08} will be useful.  More information on primitive idempotent decomposition can be found in \cite{Jespers}.

\begin{prop}\label{prop:cyclic}
Let $\xi$ be a generator of the cyclic group $C_r$, and $\om_r=\exp(2\pi i/r)$. For $1\le s\le r$, define
\begin{align}\label{eq:JofXi}
J_{r,s}=\frac{1}{r}\sum_{j=1}^r\om_r^{-sj}\xi^j.
\end{align}
Then $\{J_{r,1},\ldots,J_{r,r}\}$ form an orthogonal basis of $\mathbb{C}C_r$, and $J_{r,s}^2=J_{r,s}$ for each $1\le s\le r$. We also have
\begin{align}\label{eq:XiofJ}
\xi^s=\sum_{j=1}^r\om_r^{js}J_{r,j},  ~1 \leq s \leq r.
\end{align}
\end{prop}
\proof Define the $r\times r$ symmetric matrix $M_r$ whose $(s,j)$th entry is given by $M_r(s,j)=\frac{1}{r}\om_r^{-sj}$. It is easy to check that
$ M_rM_r^*=\frac{1}{r}I_r $, where $I_r$ denotes the identity matrix of order $r$.
Thus \eqref{eq:JofXi} is equivalent to
\begin{align}
[J_{r,1},J_{r,2},\ldots, J_{r,r}]&=[\xi,\xi^2,\ldots,\xi^r]M_r.
\end{align}
Multiplying by $rM_r^{*}$ on both sides, we obtain
\[
[\xi,\xi^2,\ldots,\xi^r]=[J_{r,1},J_{r,2},\ldots, J_{r,r}](rM_r^*),
\]
and \eqref{eq:XiofJ} follows.
\qed

Define
\begin{align}\label{eq:E}
E=\frac{1}{|\Ecal|}\sum_{\eps\in \Ecal}\eps.
\end{align}
It is easy to verify
\begin{align}
E\eps&=E, \hbox{ for each }\eps\in \Ecal, \\
E^2&=E.
\end{align}

It is well known that a finite abelian group is isomorphic to the direct product of cyclic groups. Thus we may write
\[
\Ecal\cong C_{r_1}\times C_{r_2}\times \cdots \times C_{r_f},
\]
where $C_{r_j}$ is the cyclic group of order $r_j$ and $r_1r_2\cdots r_f=|\Ecal|$.   We let  $\xi_i$ be the generator  of $C_{r_j}$  for  $1\leq i \leq f$, and  denote the $r_i$-th root of unity by
$$\om_{r_i}=\exp(2\pi i/r_i).$$

 Let $[n]$ denote the set $\{1,2,\ldots, n\}$. For convenience, we denote
\[
\Rcal:= [r_1]\times [r_2] \times \cdots \times [r_f],
\]
and
\[
\Jcal:=\{\vec{j}\in [r_1]\times [r_2] \times \cdots \times [r_f]:\vec{j}\ne \vec{r}\},
\]
where $\vec{j} = (j_1, j_2, \ldots, j_f)$ and  $\vec{r} = (r_1, r_2, \ldots, r_f)$.

Denote
 \begin{align}
 B_{\vec{s}}:=J_{r_1,s_1}\times J_{r_2,s_2} \times \cdots \times J_{r_k,s_k}. \label{eq:Bj}, \end{align}
for each $\vec{s} = (s_1, s_2, \ldots, s_f) \in [r_1]\times [r_2] \times \cdots \times [r_f]$.  It follows from Proposition~\ref{prop:cyclic} that the set
$
\{B_{\vec{s}}:\vec{s}\in [r_1]\times [r_2] \times \cdots \times [r_f]\}
$
forms an orthogonal basis of $\mathbb{C}\Ecal$ and $B_{\vec{s}}^2=B_{\vec{s}}$ for all $\vec{s}\in [r_1]\times [r_2] \times \cdots \times [r_f]$. In particular,
\[
B_{\vec{r}}=E.
\]

Next we consider some subsets of the group $\Ecal$ refined by the parameter $d$.
For type~I,  let
$$\Ecal_d =  \{ \langle f \rangle  :  f \in \Mcal_d,  f(0) \neq 0 \}$$ with $1 \leq d \leq \ell+t-1$.  For type~II,  we let
$$\Ecal_d =  \{ \langle f \rangle  :  f \in \Mcal_d\}$$ with $1 \leq d \leq \ell-1$.  We note that  $\Ecal_d = q^{d -\ell -t} \Ecal$ for $d \geq \ell +t$.

For both types,    we define
\begin{align}
c_{d,\vec{j}} &= \sum_{\eps \in \Ecal_d }\left\llbracket \eps= \prod_{i=1}^f  \xi_i^{v_i}\right\rrbracket \prod_{i=1}^f  \om_{r_i}^{v_i  j_i}.\label{eq:cdB}
\end{align}

The following result expresses $\sum_{\eps\in \Ecal_d}\eps$ in terms of the above orthogonal basis.
\begin{prop} \label{prop:basis}  Let $c_{d,\vec{j}}$ be defined in \eqref{eq:cdB}. Then we have
\begin{align}
\sum_{\eps \in \Ecal_d} \eps &=(q-1)q^{d-1}E+\sum_{\vec{j}\in \Jcal}c_{d,\vec{j}}B_{\vec{j}} ~\hbox{ for type~I and $1\le d\le \ell+t-1$}, \label{I:AofB}\\
\sum_{\eps \in \Ecal_d} \eps &=q^{d}E+\sum_{\vec{j}\in \Jcal}c_{d,\vec{j}}B_{\vec{j}} ~\hbox{ for type~II and $1\le d\le \ell-1$}, \label{II:AofB}
\end{align}
\end{prop}
\proof Since $\{E\}\cup \{B_{\vec{j}} : \vec{j}\in \Jcal\}$ is a basis, we may write
\begin{align}\label{eq:LinExp}
\sum_{\eps \in \Ecal_d} \eps&=aE+\sum_{\vec{j}\in \Jcal}a_{\vec{j}}B_{\vec{j}}
\end{align}
for some complex numbers $a$ and $a_{\vec{j}}$.

Since the basis is orthogonal and the basis elements are idempotent, we have
\begin{align}
a&=[E]\left(\sum_{\eps \in \Ecal_d} \eps\right) =|\Ecal_d|,
\end{align}
which is equal to $(q-1)q^{d-1}$ for type~I, and $q^{d}$ for type~II.

Similarly, we have
\begin{align}
a_{\vec{j}}&=[B_{\vec{j}}]\left(\sum_{\eps \in \Ecal_d} \eps \right)\nonumber\\ &=[B_{\vec{j}}]\left(\sum_{\eps \in \Ecal_d }\left\llbracket \eps= \prod_{i=1}^f  \xi_i^{v_i}\right\rrbracket \prod_{i=1}^f   \xi_i^{v_i} \right)\nonumber\\
&=[B_{\vec{j}}]\left(\sum_{\eps \in \Ecal_d }\left\llbracket \eps= \prod_{i=1}^f  \xi_i^{v_i}\right\rrbracket \prod_{i=1}^f \sum_{u_i=1}^{r_i}\om_{r_i}^{u_iv_i}J_{r_i,u_i} \right)\nonumber\\
&=\sum_{\eps \in \Ecal_d }\left\llbracket \eps= \prod_{i=1}^f  \xi_i^{v_i}\right\rrbracket \prod_{i=1}^f \om_{r_i}^{j_iv_i},\nonumber
\end{align}
which is $c_{d,\vec{j}}$ defined in \eqref{eq:cdB}.
\qed

Now we are ready to prove our main result.

\begin{thm} \label{thm:main} Let $c_{d,\vec{j}}$ be defined in \eqref{eq:cdB}, let $\tau=\ell+t-1$ for type~I and $\tau=\ell-1$ for type~II. Define
\begin{align}
P_{\vec{j}}(z)&=1+\sum_{k=1}^{\tau}c_{k,\vec{j}}z^k. \label{eq:PB}
\end{align}
With $\eps=\xi_1^{t_1}\cdots \xi_f^{t_f}$, the following hold.
\begin{itemize}
 \item[(I)] For type~I, we have
\begin{align}
\ln F&=E\ln \frac{1-z}{1-qz}+\sum_{\vec{j}\in \Jcal}B_{\vec{j}}\,\ln P_{\vec{j}}(z), \label{eq:LN1}\\
[\eps]\ln F&=\frac{1}{q-1}q^{1-\ell-t}\left(
\ln \frac{1-z}{1-qz}+\sum_{\vec{j}\in \Jcal}\prod_{i=1}^f\omega_{r_i}^{-j_it_i}\ln P_{\vec{j}}(z)\right).
\label{eq:N1}
\end{align}
 \item[(II)] For type~II,  we have
\begin{align}
\ln F&=E\ln \frac{1}{1-qz}+\sum_{\vec{j}\in \Jcal}B_{\vec{j}}\,\ln P_{\vec{j}}(z), \label{eq:LN2}\\
[\eps]\ln F&=q^{-\ell}\left(
\ln \frac{1}{1-qz}+\sum_{\vec{j}\in \Jcal}\prod_{i=1}^f\omega_{r_i}^{-j_it_i}\ln P_{\vec{j}}\right).
\label{eq:N2}
\end{align}
\end{itemize}
\end{thm}
\proof For type~I,  by  Proposition~\ref{prop:basis},   \eqref{eq:E} and definition of $\Ecal_d$, we have

\begin{align*}
\sum_{d\ge 1}z^d\sum_{\eps\in \Ecal_d}\eps &=\sum_{d=1}^{\ell+t-1}\left((q-1)q^{d-1}E+\sum_{\vec{j}\in \Jcal}c_{d,\vec{j}}B_{\vec{j}}\right)z^d+\sum_{d\ge \ell+t}(q-1)q^{d-1}z^dE \\
&=E\sum_{d\ge 1}(q-1)q^{d-1}z^d+\sum_{\vec{j}\in \Jcal}B_{\vec{j}}\sum_{d=1}^{\ell+t-1}c_{d,\vec{j}}z^d \\
&=E\frac{(q-1)z}{1-qz}+\sum_{\vec{j}\in \Jcal}B_{\vec{j}}\sum_{d=1}^{\ell+t-1}c_{d,\vec{j}}z^d.
\end{align*}
Hence
\begin{align*}
\ln F&=\sum_{m\ge 1}\frac{(-1)^{m-1}}{m}\left(\sum_{d\ge 1}z^d\sum_{\eps\in \Ecal_d}\eps \right)^m  \\
&=\sum_{m\ge 1}\frac{(-1)^{m-1}}{m}\left(E\left(\frac{(q-1)z}{1-qz}\right)^m +
\sum_{\vec{j}\in \Jcal}B_{\vec{j}}\left(\sum_{d=1}^{\ell+t-1}c_{d,\vec{j}}z^d\right)^m\right)\\
&=E\ln\left(1+ \frac{(q-1)z}{1-qz}\right)+\sum_{\vec{j}\in \Jcal}B_{\vec{j}}\ln \left(1+\sum_{d=1}^{\ell+t-1}c_{d,\vec{j}}z^d\right),
\end{align*}
which gives \eqref{eq:LN1}. Using Proposition~\ref{prop:cyclic}
and extracting the coefficient of $\eps$ from \eqref{eq:LN1}, we obtain \eqref{eq:N1}.

Similarly for type~II,  we have

\begin{align*}
\sum_{d\ge 1}z^d\sum_{\eps\in \Ecal_d}\eps &=\sum_{d=1}^{\ell-1}\left(q^dE+\sum_{\vec{j}\in \Jcal}c_{d,\vec{j}}B_{\vec{j}}\right)z^d+\sum_{d\ge \ell}q^{d}z^dE \\
&=E\sum_{d\ge 1}q^{d}z^d+\sum_{\vec{j}\in \Jcal}B_{\vec{j}}\sum_{d=1}^{\ell-1}c_{d,\vec{j}}z^d\\
&=E\frac{qz}{1-qz}+\sum_{\vec{j}\in \Jcal}B_{\vec{j}}\sum_{d=1}^{\ell-1}c_{d,\vec{j}}z^d.
\end{align*}
Therefore
\begin{align*}
\ln F&=\sum_{m\ge 1}\frac{(-1)^{m-1}}{m}\left(\sum_{d\ge 1}z^d\sum_{\eps\in \Ecal_d}\eps \right)^m  \\
&=\sum_{m\ge 1}\frac{(-1)^{m-1}}{m}\left(E\left(\frac{qz}{1-qz}  \right)^m +
\sum_{\vec{j}\in \Jcal}B_{\vec{j}}\left(\sum_{d=1}^{\ell-1}c_{d,\vec{j}}z^d\right)^m\right)\\
&=E\ln\left(1+ \frac{qz}{1-qz}\right)+\sum_{\vec{j}\in \Jcal}B_{\vec{j}}\ln \left(1+\sum_{d=1}^{\ell-1}c_{d,\vec{j}}z^d\right),
\end{align*}
which gives \eqref{eq:LN2}. Equation~\eqref{eq:N2} follows similarly as for type~I. \qed

\medskip

For the rest of the paper, we shall also use $N_n(\vec{t})$ to denote $N_n(\eps)$ etc. when $\eps=\xi_1^{t_1}\cdots\xi_f^{t_f}$
and $\vec{t} = (t_1, t_2, \ldots, t_f)$. The following corollary is immediate.
\begin{cor}\label{cor:trivial} Let $P(z)=\prod_{\vec{j}\in \Jcal}P_{\vec{j}}(z)$, and $\Scal$ be the multiset of all complex roots of $P_{\vec{j}}(z)$. Then
\begin{itemize}
\item[Type~I:]
\begin{align}
N_n(\vec{0})&=\frac{1}{q-1}q^{1-\ell-t}\left(q^n-1\right)
+\frac{1}{q-1}q^{1-\ell-t}n[z^n]\ln P(z)   \\
&=\frac{1}{q-1}q^{1-\ell-t}\left(q^n-1\right)-\frac{1}{q-1}q^{1-\ell-t}
\sum_{\rho\in \Scal}\rho^{-n}.
\end{align}
\item[Type~II:]
\begin{align}
N_n(\vec{0})&=q^{n-\ell}
+q^{-\ell}n[z^n]\ln P(z)   \\
&=q^{n-\ell}-q^{-\ell}
\sum_{\rho\in \Scal}\rho^{-n}.
\end{align}
\end{itemize}
\end{cor}
\medskip

For a polynomial $g(z)\in \mathbb{C}[z]$, we use ${\bar g}(z)$ to denote the polynomial obtained from $g(z)$ by changing all the  coefficients to their conjugate. From \eqref{eq:cdB} it is clear
\begin{align}\label{eq:conjugate}
P_{j_1,\ldots j_f}(z)={\bar P}_{r_1-j_1,\ldots,r_f-j_f}(z).
\end{align}
This equation will be used in the next two sections.

Recall that the characteristic polynomial of $\alp\in\mathbb{F}_{q^n}$ over $\Fq$ is defined as
\begin{align*}
Q_{\alp}(x)=\prod_{j=0}^{n-1}(x-\alp^{q^j}).
\end{align*}
For $\eps\in \Ecal$, define
\begin{align} \label{eq:variety}
F_q(n,\eps)=\#\{\alp\in\mathbb{F}_{q^n}:\langle Q_{\alp}\rangle=\eps\}.
\end{align}

An application of the multinomial theorem and a generalized M\"{o}bius inversion-type argument gives the enumeration of irreducible polynomials with prescribed coefficients. This equivalence follows the approach of Miers and Ruskey \cite{Miers}.
A method used by Hayes \cite{Hay65}, Hsu \cite{Hsu} and Voloch \cite{Voloch}  and others relates the enumeration of irreducible polynomials  of  degree $n$ with prescribed coefficients (equivalently,  formulae for $F_q(n,\eps)$) to the number
of points over $\mathbb{F}_{q^n}$  of certain curves defined over $\mathbb{F}_q$ whose function fields are subfields of the so-called cyclotomic functions fields.
Granger \cite{Gra19} studied $F_q(n,\eps)$ for the type II equivalence class in detail and used it to count irreducible polynomials with prescribed leading coefficients. Through a transformation of the problem of counting the number of elements of $\mathbb{F}_{q^n}$ with prescribed traces to the problem of counting the number of elements for which linear combinations of the trace functions evaluate to $1$,  he reduced the varieties in \eqref{eq:variety} to Artin-Schreier curves of smaller genus and then computed the corresponding zeta functions using Lauder-Wan algorithm \cite{LW}.

Our next theorem connects $N_n(\eps)$ with $F_q(n,\eps)$. This gives an alternative way of computing these zeta functions.
\begin{thm}\label{thm:Zeta}
For both types, we have
\begin{equation} \label{eq:NF}
N_n(\eps)=F_q(n,\eps).
\end{equation}
Moreover, with $\eps=\xi_1^{t_1}\cdots\xi_f^{t_f}$, we have
\begin{itemize}
\item[(I)] the logarithm of the  Hasse-Weil zeta function of $N_n(\eps)$ for type~I is equal to
\begin{align*}
\frac{1}{q-1}q^{1-\ell-t}\left(\ln \frac{1-z}{1-qz}+\sum_{ \vec{j}\in \Jcal}\prod_{i=1}^f\omega_{r_i}^{-j_it_i}\ln P_{\vec{j}}(z)\right),\quad and 
\end{align*}
\item[(II)]
the logarithm of the Hasse-Weil zeta function of $N_n(\eps)$ for type~II is equal to
\begin{align*}
q^{-\ell}\left(\ln \frac{1}{1-qz}+\sum_{ \vec{j}\in \Jcal}\prod_{i=1}^f\omega_{r_i}^{-j_it_i}\ln P_{\vec{j}}(z)\right).
\end{align*}
\end{itemize}
\end{thm}

\proof
 Equation~\eqref{eq:Irre}  can be written as
\begin{align*}
\ln F=\sum_{m\ge1}\sum_{k\mid m}\sum_{f\in \Ical_{m/k}}\frac{\langle f^k \rangle}{k}z^m.
\end{align*}

For $k\mid n$, each $f\in \Ical_{n/k}$ has $n/k$ roots $\alp\in\mathbb{F}_{q^n}$. For each $\alp\in\mathbb{F}_{q^n}$ with $f(\alp)=0$,  we have $Q_\alp=f^k$.

Therefore,
\begin{align*}
n[z^n]\ln F&=\sum_{k\mid n}\sum_{f\in \Ical _{n/k}}\frac{n}{k}{\langle f^k \rangle}\\
&=\sum_{k\mid n}\sum_{f\in \Ical_{n/k}}\sum_{\alp\in\mathbb{F}_{q^n}:f(\alp)=0}{\langle f^k \rangle}\\
&=\sum_{k\mid n}\sum_{f\in \Ical_{n/k}}\sum_{\alp\in\mathbb{F}_{q^n}:f(\alp)=0}{\langle Q_\alp \rangle}.
\end{align*}

Each $\alp\in\mathbb{F}_{q^n}$ is the root of some unique irreducible polynomial $f$ over $\Fq$ of degree $n/k$ with $k\mid n$ and $f^k=Q_\alp$, hence
we have
\begin{align*}
n[z^n]\ln F&=\sum_{\alp\in\mathbb{F}_{q^n}}\langle Q_\alp \rangle.
\end{align*}
which gives \eqref{eq:NF}.
The claim about the zeta functions follows from the definition of zeta function and Theorem~\ref{thm:main}.
\qed

\section{Type~I examples} \label{TypeI}

Let $\ell \geq 1$ and $t\geq 1$ be fixed integers.   In this section, we use Theorem~\ref{thm:main} to derive formulas for the number of monic irreducible polynomials over $\fq$ with prescribed $\ell$ leading coefficients and $t$ ending coefficients.
Throughout this section, we let $\tau=\ell+t-1$.

\begin{thm}\label{thm:typeI}
Let $q$ be a prime power, $\ell$ and $t$ be positive integers.  Let $\xi_1, \ldots, \xi_f$ be generators of
type~I group $\Ecal$ with order $r_1, \ldots, r_f$ respectively.   Let $\omega_{r_i} = exp(2\pi i/r_i)$ for $1\leq i \leq f$.  Let $c_{d, \vec{j}}$ be defined by  \eqref{eq:cdB}
and $P_{\vec{j}}(z)$ is defined in \eqref{eq:PB}.
 Suppose each polynomial $P_{\vec{j}}(z)$ is factored into linear factors, and let $\Scal_{\vec{j}}$ be the multiset of all the complex roots of $P_{\vec{j}}(z)$. Then the number of monic irreducible polynomials over $\mathbb{F}_q$ with the prescribed first $\ell$ coefficients  and the last $t$ coefficients    $\eps=\xi_1^{t_1}\cdots \xi_f^{t_f}$  is
\begin{align}
I_d(\vec{t})=\frac{1}{d}\sum_{k|d}\mu(k)\sum_{\vec{s}\in \Rcal}
N_{d/k}(\vec{s})\left\llbracket k\vec{s}\equiv \vec{t} ~(\bmod{~\vec{r}})\right\rrbracket,
\end{align}
where
\begin{align}
N_n(\vec{s})&=\frac{1}{q-1}q^{-\tau}\left(q^n-1\right)+\frac{1}{q-1}q^{-\tau}\sum_{\vec{j}\in \Jcal}
 \prod_{i=1}^{f}\om_{r_i}^{-j_is_i}n[z^n]\ln P_{\vec{j}}(z)\label{eq:factorI1}\\
 &=\frac{1}{q-1}q^{-\tau}\left(q^n-1\right)-\frac{1}{q-1}q^{-\tau}
\sum_{\vec{j}\in \Jcal}\prod_{i=1}^f\omega_{r_i}^{-j_is_i}\sum_{\rho\in \Scal_{\vec{j}}}\rho^{-n}.\label{eq:factorI}
\end{align}
\end{thm}

\begin{proof}
Equation~\eqref{eq:factorI1} follows immediately from Theorem~\ref{thm:main}.

Now \eqref{eq:factorI} follows from
\[
\ln P_{\vec{j}}(z)=\sum_{\rho\in \Scal_{\vec{j}}}\ln(1-z/\rho),
\]
and the expansion
\[
\ln(1-z/\rho)=-\sum_{n\ge 1}\rho^{-n}z^n/n.
\]
\qed
\end{proof}

\begin{example}
Let $q=2,\ell=t=2$.  In this case, $|\Ecal|=2^3$. The generators are $\xi_1=\langle x+1\rangle$ and $\xi_2=\langle x^4+x+1\rangle$ of orders 4 and 2, respectively. We have $\Ecal_1=\{\xi_1\}$, $\Ecal_2=\{\xi_1^2,\xi_1^3\}$, $\Ecal_3=\{\langle 1\rangle,\xi_1\xi_2,\xi_1^2\xi_2,\xi_1^3\}$. Using \eqref{eq:cdB},  we have
\begin{align*}
c_{1, (j_1,j_2)}&=i^{j_1},\\
c_{2, (j_1,j_2)}&=i^{2j_1}+i^{3j_1},\\
c_{3, (j_1,j_2)}&=1+i^{j_1}(-1)^{j_2}+i^{2j_1}(-1)^{j_2}+i^{3j_1},\\
P_{j_1,j_2}(z)&=1+c_{1, (j_1,j_2)}z+c_{2, (j_1,j_2)}z^2+c_{3, (j_1,j_2)}z^3,
\end{align*}
and hence   $P_{1,j}(z)={\bar P}_{3,j}(z)$,  and
\begin{align*}
P_{1,1}(z)&=1+iz-(1+i)z^2+2(1-i)z^3,\\
P_{1,2}(z)&=(1-z)(1+(1+i)z),\\
P_{2,1}(z)&=P_{2,2}(z)=1-z,\\
P_{4,1}(z)&=1+z+2z^2,\\
P_{1,1}(z)P_{3,1}(z)&=1-z^2+2z^3-2z^4+8z^6,\\
P_{1,2}(z)P_{3,2}(z)&=(1-z)^2(1+2z+2z^2).
\end{align*}

For $\eps=\xi_1^{2s_1}\xi_2^{s_2}$, we may use \eqref{eq:conjugate} to combine the conjugate pairs to obtain
\begin{align}
N_n(\xi_1^{2s_1}\xi_2^{s_2})&=\frac{1}{8}\left(2^n-1\right)
-\frac{1}{8}\left(1+(-1)^{s_2}+2(-1)^{s_1}\right)\nonumber\\
&~+(-1)^{s_2}\frac{n}{8}[z^n]\ln(1+z+2z^2)\nonumber\\
&~+(-1)^{s_1}\frac{n}{8}[z^n]\ln (1+2z+2z^2)\\
&~+(-1)^{s_1+s_2}\frac{n}{8}[z^n]\ln (1-z^2+2z^3-2z^4+8z^6). \nonumber
\end{align}

Using Maple to expand the logarithmic functions, we obtain Table~\ref{table222}.
\end{example}

\begin{table}
\begin{center}
\begin{tabular}{|r|r|r|r|r|}
\hline
$n$&$N_n(\xi_1^0\xi_2^0)$&$N_n(\xi_2)$&$N_n(\xi_1^2)$&$N_n(\xi_1^2\xi_2)$\\
\hline
1 & 0& 0&0 &0 \\
\hline
2 &0 &0 & 1& 0\\
\hline
3 & 0& 0 & 0& 3\\
\hline
4 &1&4 & 2& 0\\
\hline
5 &5& 0 & 5& 5\\
\hline
6&9 & 6 & 4& 12\\
\hline
7 & 21&14 &7 & 21\\
\hline
8 &31 & 24& 40& 32 \\
\hline
9 & 63& 72 & 63& 57\\
\hline
10 & 125& 130 &116 & 140\\
\hline
11 & 253& 242 & 275& 253\\
\hline
12 & 523& 532 & 512&480\\
\hline
13 & 923& 1092 & 1079&1001\\
\hline
14 & 2065&2030 &2052 & 2044\\
\hline
15 & 4145& 4110 &4115 &4013\\
\hline
16 & 8143& 8112 &8128 & 8384\\
\hline
17 &16303& 16592 &16439 &16201\\
\hline
18 &33093& 32442 &32692 &32844\\
\hline
19 & 65493&65322 &65379 & 65949\\
\hline
20 & 131731 & 130924&130112 & 131520\\
\hline
\end{tabular}
\end{center}
\caption{Values of $N_n(\xi_1^{2s_1}\xi_2^{s_2})$ for $q=2,\ell =2, t=2$.}
\label{table222}
\end{table}

\begin{example}
Let $q=3$, $\ell=2$ and $t=1$.  In this case,  $|\Ecal|=18$. The generators are $\xi_1=\langle x+1\rangle$ and $\xi_2=\langle x+2\rangle$ of orders 3 and 6, respectively. We have
\begin{align}
\Ecal_1&=\{\xi_1,\xi_2\},\nonumber\\
\Ecal_2&=\{\xi_2^2,\xi_1\xi_2,\xi_1\xi_2^5,\xi_1^2,\xi_1^2\xi_2,\xi_1^2\xi_2^2\}.
\end{align}

Using \eqref{eq:cdB},  we have
\begin{align*}
c_{1, (j_1,j_2)}&=\om_3^{j_1}+\om_6^{j_2},\\
c_{2, (j_1,j_2)}&=\om_6^{2j_2}+\om_3^{j_1}\om_6^{j_2}+\om_3^{j_1}\om_6^{5j_2}+\om_3^{2j_1}+\om_3^{2j_1}\om_6^{j_2}
+\om_3^{2j_1}\om_6^{2j_2},\\
P_{j_1,j_2}(z)&=1+c_{1, (j_1,j_2)}z+c_{2,(j_1,j_2)}z^2.
\end{align*}
We have $P_{1,j}(z)={\bar P}_{2,6-j}(z)$, $P_{3,j}(z)={\bar P}_{3,6-j}(z)$,
and
\begin{align*}
P_{1,1}(z)&=1+i\sqrt{3}z,\\
P_{1,2}(z)&=(1-z)(1+i\sqrt{3}z),\\
P_{1,3}(z)&=(1+i\sqrt{3}z)\left(1+i\sqrt{3}\om_3z\right),\\
P_{1,4}(z)&=1-z,\\
P_{1,5}(z)&=1-3z^2,\\
P_{1,6}(z)&=(1-z)\left(1-i\sqrt{3}\om_3z\right),\\
P_{3,1}(z)&=(1+i\sqrt{3}z)\left(1-i\sqrt{3}\om_6z\right),\\
P_{3,2}(z)&=(1-z)\left(1-i\sqrt{3}\om_3z\right),\\
P_{3,3}(z)&=1.
\end{align*}
Applying \eqref{eq:factorI}  and combining conjugate pairs, we obtain
\begin{align}
N_n(\xi_1^{t_1}\xi_2^{t_2})
&=\frac{1}{18}\left(3^{n}-1\right)
-\frac{2}{18}3^{n/2}\Re\left(\om_6^{-2t_1-t_2}(-i)^n\right)
-\frac{2}{18}3^{n/2}\Re\left(\om_3^{-t_1-t_2}(-i)^n\right)\nonumber\\
&~~-\frac{2}{18}\Re\left(\om_3^{-t_1-t_2}\right)
-\frac{2}{18}3^{n/2}\Re\left(\om_6^{-2t_1-3t_2}\left((-i)^n+(-i\om_3)^n\right)\right)
-\frac{2}{18}\Re\left(\om_3^{-t_1-2t_2}\right)\nonumber\\
&~~-\frac{4\llbracket 2\mid n\rrbracket}{18}3^{n/2}\Re\left(\om_6^{-2t_1-5t_2}\right)
-\frac{2}{18}3^{n/2}\Re\left(\om_3^{-t_1}(i\om_3)^n\right)
-\frac{2}{18}\Re\left(\om_3^{-t_1}\right)\nonumber\\
&~~-\frac{2}{18}3^{n/2}\Re\left(\om_6^{-t_2}\left((-i)^n+(i\om_6)^n\right)\right)
-\frac{2}{18}\Re\left(\om_6^{-2t_2}\right)
-\frac{2}{18}3^{n/2}\Re\left(\om_6^{-2t_2}(i\om_3)^n\right)\nonumber\\
&=\frac{1}{18}\left(3^{n}-1\right)-\frac{4\llbracket 2\mid n\rrbracket}{18}3^{n/2}\cos\left(\frac{\pi(2t_1+5t_2)}{3}\right)  \\
&~~-\frac{2}{18}\left(\cos\left(\frac{2\pi(t_1+t_2)}{3}\right)+
\cos\left(\frac{2\pi (t_1+2t_2)}{3}\right)+\cos\left(\frac{2\pi t_1}{3}\right)+\cos\left(\frac{2\pi t_2}{3}\right) \right)\nonumber\\
&~~-\frac{2}{18}3^{n/2}\left(\cos\left(\frac{3n\pi}{2}-\frac{(2t_1+t_2)\pi}{3}\right)+
\cos\left(\frac{3n\pi}{2}-\frac{2(t_1+t_2)\pi}{3}\right)   \right)\nonumber\\
&~~-\frac{2}{18}3^{n/2}(-1)^{t_2}\left(\cos\left(\frac{3n\pi}{2}-\frac{2t_1\pi}{3}\right)+
\cos\left(\frac{3n\pi}{2}+\frac{2(n-t_1)\pi}{3}\right)   \right)\nonumber\\
&~~-\frac{2}{18}3^{n/2}\left(\cos\left(\frac{n\pi}{2}+\frac{2(n-t_1)\pi}{3}\right)+
\cos\left(\frac{3n\pi}{2}-\frac{t_2\pi}{3}\right)   \right)\nonumber\\
&~~-\frac{2}{18}3^{n/2}\left(\cos\left(\frac{n\pi}{2}+\frac{(n-t_2)\pi}{3}\right)+
\cos\left(\frac{n\pi}{2}+\frac{2(n-t_2)\pi}{3}\right)   \right).\nonumber
\end{align}

Using Maple, we obtain Tables~\ref{table321I}, \ref{table321II} and \ref{table321III}.

\end{example}

\begin{table}
\begin{center}
\begin{tabular}{|r|r|r|r|r|r|r|}
\hline
$n$&$N_n(\xi_2^0)$&$N_n(\xi_2^1)$&$N_n(\xi_2^2)$&$N_n(\xi_2^3)$&$N_n(\xi_2^4)$&$N_n(\xi_2^5)$\\
\hline
1 & 0&1& 0& 0&  0 &0\\
\hline
2 &0 &0 &1&0 &0 & 0 \\
\hline
3 & 1& 0& 0& 1 &3& 3\\
\hline
4 &0&4 & 8&8&5 & 4\\
\hline
5 &10&15&15& 10 &15&6  \\
\hline
6&58 &36& 45& 40& 45 &36\\
\hline
7 & 112&99& 126&112& 126 &126\\
\hline
8 &328 &360&369& 400 &396 &360\\
\hline
9 & 1093&1134& 1134& 1093&1053 &1053 \\
\hline
10 & 3280& 3240& 3240&3280 &3321 &3240 \\
\hline
11 & 9922& 9801&9801& 9922&9801 &10044\\
\hline
12 & 28714& 29484& 29565& 29848&29565& 29484 \\
\hline
13 & 88816& 89181 &88452&88816&88452&88452 \\
\hline
14 & 265720 &265356& 266085&265720&265356& 265356 \\
\hline
15 & 797161& 796068&796068&797161&798255&798255 \\
\hline
16 &2388568 &2391120&2394036&  2394400 &2391849& 2391120\\
\hline
17 &7172266 & 7175547&7175547& 7172266&7175547& 7168986\\
\hline
18 &21536482& 21520080& 21526641& 21523360&21526641&21520080 \\
\hline
19 &64563520&64553679&64573362&64563520& 64573362&64573362 \\
\hline
20 & 193684000 & 193706964&193713525& 193736488&193733208& 193706964\\
\hline
\end{tabular}
\end{center}
\caption{Values of $N_n(\xi_2^j)$ for $q=3,\ell =2, t=1$.}
\label{table321I}
\end{table}

\begin{table}
\begin{center}
\begin{tabular}{|r|r|r|r|r|r|r|}
\hline
$n$&$N_n(\xi_1)$&$N_n(\xi_1\xi_2)$&$N_n(\xi_1\xi_2^2)$&$N_n(\xi_1\xi_2^3)$&$N_n(\xi_1\xi_2^4)$&$N_n(\xi_1\xi_2^5)$\\
\hline
1 & 1&0& 0& 0&  0 &0\\
\hline
2 &0 &0 &0&0 &0 & 2 \\
\hline
3 & 0& 3& 0& 3 &3& 3\\
\hline
4 &5&4 & 2&4&2 & 4\\
\hline
5 &15&15&15& 15  &15&15  \\
\hline
6&45 &36& 27& 36& 36 &54\\
\hline
7 & 99&126& 126&126& 126 &126\\
\hline
8 &396 &360&341& 360 &396 &360\\
\hline
9 & 1134&1053& 1134& 1053&1053 &1053 \\
\hline
10 & 3321& 3240& 3240&3240 &3402 &3402 \\
\hline
11 & 9801& 9801&9801& 9801&9801 &9801\\
\hline
12 & 29565& 29484& 29565&29484&29808& 29484 \\
\hline
13 & 89181& 88452 &88452&88452&88452&88452 \\
\hline
14 & 265356 &265356& 265356&265356 &265356& 266814 \\
\hline
15 &  796068& 798255&796068&798255&798255&798255 \\
\hline
16 &2391849 &2391120& 2389662 &  2391120 & 2389662& 2391120\\
\hline
17 & 7175547& 7175547&7175547& 7175547&7175547& 7175547\\
\hline
18 &21526641& 21520080& 21513519& 21520080&21520080&21533202\\
\hline
19 &64553679&64573362&64573362&64573362& 64573362&64573362 \\
\hline
20 & 193733208 & 193706964  & 193693842& 193706964&193733208& 193706964\\
\hline
\end{tabular}
\end{center}
\caption{Values of $N_n(\xi_1\xi_2^j)$ for $q=3,\ell =2, t=1$.}
\label{table321II}
\end{table}

\begin{table}
\begin{center}
\begin{tabular}{|r|r|r|r|r|r|r|}
\hline
$n$&$N_n(\xi_1^2)$&$N_n(\xi_1^2\xi_2)$&$N_n(\xi_1^2\xi_2^2)$&$N_n(\xi_1^2\xi_2^3)$&$N_n(\xi_1^2\xi_2^4)$&$N_n(\xi_1^2\xi_2^5)$\\
\hline
1 &0&0& 0& 0&  0 &0\\
\hline
2 &1 &2 &2&0 &0 & 0 \\
\hline
3 & 3& 0& 0& 0 &3& 0\\
\hline
4 &8&4 & 8&4&2 & 4\\
\hline
5 &6&15&15& 15 &15&15  \\
\hline
6&45 &54& 36& 36&27 &36\\
\hline
7 & 126&  126 & 126&126& 126 &126\\
\hline
8 &369 &360&342& 360 &342 &360\\
\hline
9 & 1053& 1134 & 1134& 1134&1053 &1134 \\
\hline
10 & 3240& 3402& 3240&3240 &3240 &3240 \\
\hline
11 & 10044& 9801&9801& 9801&9801 &9801\\
\hline
12 & 29565& 29484& 29808&29484 &29565 & 29484 \\
\hline
13 & 88452& 88452 &88452&88452&88452&88452 \\
\hline
14 & 266085 &266814& 266814&265356 &265356& 265356 \\
\hline
15 & 798255& 796068 &796068&796068&798255&796068 \\
\hline
16 &2394036 &2391120&2394036&  2391120 & 2389662& 2391120\\
\hline
17 &7168986& 7175547&7175547& 7175547&7175547& 7175547\\
\hline
18 &21526641& 21533202 & 21520080& 21520080&21513519&21520080\\
\hline
19 &64573362&64573362&64573362&64573362& 64573362&64573362 \\
\hline
20 & 193713525 & 193706964&193693842& 193706964&193693842& 193706964\\
\hline
\end{tabular}
\end{center}
\caption{Values of $N_n(\xi_1^2\xi_2^j)$ for $q=3,\ell =2, t=1$.}
\label{table321III}
\end{table}

\section{Type~II examples} \label{TypeII}

The number of irreducible polynomials with prescribed leading coefficients (i.e., trace and subtraces) are treated in detail by Granger \cite{Gra19}. Our generating function approach is connected to Granger's approach through \eqref{eq:NF}. We would like to point out that our Theorem~\ref{thm:Zeta} gives an alternative way of computing $N_d(\eps)$, and this will be demonstrated through examples below.

We first note that the group $\Ecal$ is an abelian group of order $q^{\ell}=p^{r\ell}$, and thus it is a direct product of cyclic $p$-groups. For these type II groups, for any $q$, $\ell$, we suppose the generators of the group $\Ecal$ are
$\xi_1, \xi_2, \ldots, \xi_f$, with order $r_i$,  $1\leq i \leq f$ respectively.   Here $r_1 \cdots r_f = q^{\ell}$.  Let $\tau = \ell -1$.


The following result is analogous to Theorem~\ref{thm:typeI}. Its proof is essentially the same as that of Theorem~\ref{thm:typeI}.

\begin{thm} \label{thm:typeII}
Let $q$ be a prime power and $\ell$  be  a positive integer.  Let $\xi_1, \ldots, \xi_f$ be generators of
type II group $\Ecal$ with order $r_1, \ldots, r_f$ respectively.    Let $\omega_{r_i} = exp(2\pi i/r_i)$ for $1\leq i \leq f$.  Let  $c_{d, \vec{j}}$ be defined as  in  \eqref{eq:cdB} and $P_{\vec{j}}(z)$ is defined in \eqref{eq:PB}.
Suppose each polynomial $P_{\vec{j}}(z)$ is factored into linear factors, and let $\Scal_{\vec{j}}$ be the multiset of all the complex roots of $P_{\vec{j}}(z)$.
For
$\eps = \xi_1^{t_1} \cdots  \xi_f^{t_f}$, we shall use $I_d(\vec{t})$ to denote $I_d(\eps)$ and so on.
Then the number of monic irreducible polynomials over $\mathbb{F}_q$ with  prescribed $\ell$ leading coefficients   $\eps=\xi_1^{t_1}\cdots \xi_f^{t_f}$  is
\begin{align}
I_d(\vec{t})=\frac{1}{d}\sum_{k|d}\mu(k)\sum_{\vec{s}\in \Rcal}
N_{d/k}(\vec{s})\left\llbracket k\vec{s}\equiv \vec{t} ~(\bmod{~\vec{r}})\right\rrbracket,
\end{align}
where
\begin{align}
N_n(\vec{t})&=q^{n-\ell}+q^{-\ell}\sum_{ \vec{j}\in \Jcal}\prod_{i=1}^f\omega_{r_i}^{-j_it_i}n[z^n]\ln P_{\vec{j}}(z)    \\
&=q^{n-\ell}-q^{-\ell}\sum_{ \vec{j}\in \Jcal}\prod_{i=1}^f\omega_{r_i}^{-j_it_i}\sum_{\rho\in\Scal_{\vec{j}}} \rho^{-n}.\label{eq:factorII}
\end{align}
\end{thm}

The following lemma is useful before working out some examples.

\begin{lemma}\label{lemma:generator}
Let $q=p$ be a prime number. The generators of $\Ecal$ are $\langle x^j +1 \rangle$ for all $j$'s such that  $(p, j) =1$ and $j < \ell$.
\end{lemma}

\begin{proof}
Obviously,  the  order of  $\langle x^j +1 \rangle$ is   $p^{s_i}$ such that
 $s_j$ is the smallest  positive integer such that $jp^{s_j} > \ell$. In another word, $p^{s_j-1}$ is the largest power of $p$ such that $j p^{s_j-1} \leq \ell$.  Hence $\langle x^{jp^{t}} +1 \rangle$ are all distinct for $t=1, \ldots, s_j-1$. Since each integer in $[\ell]$ can be written uniquely as $jp^{t_j}$
 for some $j$ and $t_j$ satisfying $j<\ell$, $(j,p)=1$ and $t_j<s_j$, we have $\sum_{(p, j)=1, j < \ell} s_j = \ell$.  Moreover,   each of  $\langle x^j +1 \rangle$  generates a different subgroup of order $p^{s_j}$.  \qed
\end{proof}

\medskip
\begin{example}\label{ex:Q2L3} Consider $q=2,\ell=3$.
This is also treated in \cite{YucMul04} and some complicated expressions are given there. By Lemma~\ref{lemma:generator},   the group $\Ecal$ is isomorphic to $C_4\times C_2$ where $C_4$ and $C_2$ are generated  by $\xi_1$ and $\xi_2$, where $\xi_1=\langle x+1\rangle$  has order 4 and $\xi_2=\langle x^3+1\rangle$ has order 2.  In this case,
$\Ecal_1 = \{ \langle 1\rangle, \xi_1 \}$ and $\Ecal_2 = \{ \langle 1\rangle, \xi_1, \xi_1^2, \xi_1^3 \xi_2 \}$.

Using \eqref{eq:cdB},  we have
\begin{align*}
c_{1, \vec{j}}&=1+i^{j_1},\\
c_{2, \vec{j}}&=1+i^{j_1}+i^{2j_1}+i^{3j_1}(-1)^{j_2},\\
P_{j_1,j_2}(z)&=1+c_{1, \vec{j}}z+c_{2, \vec{j}}z^2.
\end{align*}

We have $P_{j_1,j_2}(z)={\bar P}_{4-j_1,j_2}(z)$, and
\begin{align*}
P_{1,1}(z)&=1+(1+i)z+2iz^2,&\\
P_{1,2}(z)&=1+(1+i)z,&\\
P_{2,1}(z)&=1+2z^2,&(\delta_{2,2})\\
P_{2,2}(z)&=1,&\\
P_{4,1}(z)&=1+2z+2z^2, &(\delta_{2,1})\\
P_{1,2}(z)P_{3,2}(z)&=1+2z+2z^2,&(\delta_{2,1}) \\
P_{1,1}(z)P_{3,1}(z)&=1+2z+2z^2+4z^3+4z^4.&(\delta_{4,1})
\end{align*}
We note that the polynomials $\delta_{i,j}(X)$ in \cite[Theorem~6]{Gra19} are the reciprocals of our corresponding polynomials. For comparison purpose, we list the corresponding $\delta_{i, j}$ on the right side.

When the exponent of $\xi_1$ is even, we may multiply conjugate pairs together to obtain (using Granger's notation)
\begin{align}
N_n(\xi_1^{2s_1}\xi_2^{s_2})
&=2^{n-3}+(-1)^{s_2}\frac{n}{8}[z^n]\ln(1+2z^2)\nonumber\\
&~+\left((-1)^{s_2}+(-1)^{s_1}\right)\frac{n}{8}[z^n]\ln(1+2z+2z^2)\nonumber\\
&~+(-1)^{s_1+s_2}\frac{n}{8}[z^n]\ln (1+2z+2z^2+4z^3+4z^4)\nonumber\\
&=2^{n-3}-\frac{\llbracket 2\mid n\rrbracket}{4}(-1)^{s_2} (-2)^{n/2}\\
&~-\frac{1}{8}\left((-1)^{s_2}+(-1)^{s_1}\right)\rho_n(\delta_{2,1})-\frac{1}{8}(-1)^{s_1+s_2}\rho_n(\delta_{4,1}).\nonumber
\end{align}
More explicitly, we have
\begin{align*}
N_n(\xi_1^{0}\xi_2^{0})&=F_2(n,0,0,0)\\
&=2^{n-3}-\frac{\llbracket 2\mid n\rrbracket}{4}(-2)^{n/2}-\frac{1}{8}\left(2\rho_n(\delta_{2,1})+\rho_n(\delta_{4,1})\right),  \\
N_n(\xi_1^{0}\xi_2^{1})&=F_2(n,0,0,1)\\
&=2^{n-3}+\frac{\llbracket 2\mid n\rrbracket}{4}(-2)^{n/2}+\frac{1}{8}\rho_n(\delta_{4,1}),  \\
N_n(\xi_1^{2}\xi_2^{0})&=F_2(n,0,1,0)\\
&=2^{n-3}-\frac{\llbracket 2\mid n\rrbracket}{4}(-2)^{n/2}+\frac{1}{8}\rho_n(\delta_{4,1}),  \\
N_n(\xi_1^{2}\xi_2^{1})&=F_2(n,0,1,1)\\
&=2^{n-3}+\frac{\llbracket 2\mid n\rrbracket}{4}(-2)^{n/2}-\frac{1}{8}\left(-2\rho_n(\delta_{2,1})+\rho_n(\delta_{4,1})\right).
\end{align*}
This immediately implies the corresponding four expression in \cite[Theorem~6]{Gra19}. Moreover, we obtain expressions for even degrees.

Since we have simple expressions for all the roots, we may obtain a more compact expression for $N_n(\xi_1^{t_1}\xi_2^{t_2})$ as follows.
\begin{align*}
N_n(\xi_1^{t_1}\xi_2^{t_2})&=2^{n-3}-2^{-3}\sum_{ (j_1,j_2)\ne (4,2)}i^{-j_1t_1}(-1)^{-j_2t_2}\sum_{\rho\in \Scal_{j_1,j_2}}\rho^{-n}\\
&=2^{n-3}+
 \frac{ (-1)^{n-1}}{4}\Re\left({i}^{-t_1}(1+i)^n\right)\\
&+\frac{1}{8}(-1)^{t_2}\sum_{j=1}^4{i}^{-jt_1}f_n(j),
\end{align*}
where
\begin{align}
f_n(j)&=n[z^n]\ln P_{j,1}(z).
\end{align}
Note that
\begin{align*}
f_n(1)&=n[z^n]\ln(1+(1+i)z+2iz^2)
\\&=n[z^n]\ln\left(\frac{1-(1+i)^3z^3}{1-(1+i)z}\right)
\\&=n[z^n]\left(\ln\left(\frac{1}{1-(1+i)z}\right)-\ln\left(\frac{1}{1-(1+i)^3z^3}\right)\right)
\\&=(1+i)^n(1-3\llbracket 3\mid n \rrbracket),\\
f_n(2)&=n[z^n]\ln(1+2z^2)=2(-1)(-2)^{n/2}\llbracket 2\mid n\rrbracket,
\\f_n(3)&=\Bar{f_n}(1)=(1-i)^n(1-3\llbracket 3\mid n \rrbracket),\\
f_n(4)&=n[z^n]\ln(1+2z+2z^2)
\\&=n[z^n]\ln((1+(1+i)z)(1+(1-i)z))
\\&=n[z^n](\ln(1+(1+i)z)+\ln(1+(1-i)z))
\\&=2(-1)^{n-1}\Re\left((1+i)^n\right).
\end{align*}
Thus
\begin{align}
N_n(\xi_1^{t_1}\xi_2^{t_2})&=2^{n-3}+\frac{ 1}{4}(-1)^{n-1}\Re\left({i}^{-t_1}(1+i)^n\right)\nonumber\\
&~~+\frac{\llbracket 2\mid n\rrbracket}{4}(-1)^{t_2+t_1-1}(-2)^{n/2}\nonumber\\
&~~+\frac{1-3\llbracket 3\mid n\rrbracket}{4}(-1)^{t_2}\Re\left(i^{-t_1}(1+i)^n\right)\nonumber\\
&~~+\frac{1}{4}(-1)^{t_2+n-1}\Re\left((1+i)^n\right)\nonumber\\
&=2^{n-3}+\frac{\llbracket 2\mid n\rrbracket}{4}(-1)^{t_2+t_1-1}(-2)^{n/2}\label{eq:ex2}\\
&~~+\frac{1}{4}(-1)^{t_2+n-1}2^{n/2}\cos\left(\frac{n\pi}{4}\right)\nonumber\\
&~~+\left(\frac{1}{4}(-1)^{n-1}+\frac{1-3\llbracket 3\mid n\rrbracket}{4  }(-1)^{t_2} \right)2^{n/2}\cos\left(\frac{n\pi}{4}-\frac{t_1\pi}{2}\right).\nonumber
\end{align}
Expression~\eqref{eq:ex2} gives a complete solution to the case $q=2$ and $\ell=3$, which improves \cite[Theorem~6]{Gra19}.
\end{example}

\begin{example}\label{ex:Q2L4} Consider $q=2,\ell=4$. By Lemma~\ref{lemma:generator},
The generators are $\xi_1=\langle x+1\rangle$  and $\xi_2=\langle x^3+1\rangle$ of orders 8 and 2, respectively. Hence
\begin{align*}
\Ecal_1 &= \{ \langle 1\rangle, \xi_1 \},\\
\Ecal_2 &=\Ecal_1 \cup \{\xi_1^2,\xi_1^7\xi_2\},\\
\Ecal_3 &= \Ecal_2 \cup \{\xi_1^2\xi_2, \xi_1^3, \xi_1^5\xi_2 ,\xi_2\}.
\end{align*}

Using \eqref{eq:cdB},  we have
\begin{align*}
c_{1, (j_1,j_2)}&=1+\om_8^{j_1},\\
c_{2, (j_1,j_2)}&=c_{1, (j_1,j_2)}+\om_8^{2j_1}+\om_8^{7j_1}(-1)^{j_2},\\
c_{3, (j_1,j_2)}&=c_{2, (j_1,j_2)}+\om_8^{2j_1}(-1)^{j_2}+\om_8^{3j_1}+\om_8^{5j_1}(-1)^{j_2}+(-1)^{j_2},\\
P_{j_1,j_2}(z)&=1+c_{1, (j_1,j_2)}z+c_{2,(j_1,j_2)}z^2+c_{3,(j_1,j_2)}z^3.
\end{align*}

Using Maple, we obtain (for comparison purpose, we list the corresponding $\delta_{j, k}$ from \cite{Gra19})
\begin{align*}
P_{1,1}(z)&=1+\frac{2+(1+i)\sqrt{2}}{2}z+(1+i(1+\sqrt{2}))z^2+2i\sqrt{2}z^3,&\\
P_{2,1}(z)&=1+(1+i)z+2iz^2,&\\
P_{3,1}(z)&=1+\frac{2+(-1+i)\sqrt{2}}{2}z+(1-i(1+\sqrt{2}))z^2+2i\sqrt{2}z^3,&\\
P_{4,1}(z)&=1+2z^2,&(\delta_{2,2})\\
P_{8,1}(z)&=1+2z+2z^2,&(\delta_{2,1})\\
P_{1,2}(z)&=1+\frac{2+(1+i)\sqrt{2}}{2}z+(1+i+\sqrt{2})z^2+(2+2i)z^3,&\\
P_{2,2}(z)&=1+(1+i)z,&\\
P_{3,2}(z)&=1+\frac{2+(-1+i)\sqrt{2}}{2}z+(-1+i+\sqrt{2})z^2+(2-2i)z^3,&\\
P_{4,2}(z)&=1,&\\
P_{2,1}(z)P_{6,1}(z)&=1+2z+2z^2+4z^3+4z^4,&(\delta_{4,1})\\
P_{2,2}(z)P_{6,2}(z)&=1+2z+2z^2,&(\delta_{2,1})
\end{align*}
and
\begin{align*}
&~~~P_{1,1}(z)P_{7,1}(z)P_{3,1}(z)P_{5,1}(z)&\\
&=(16z^8 + 8z^6 + 8z^5 + 2z^4 + 4z^3 + 2z^2 + 1)(2z^2 + 2z + 1)^2,&(\delta_{8,2},\delta_{2,1}^2)\\
&~~~P_{1,2}(z)P_{7,2}(z)P_{3,2}(z)P_{5,2}(z)&\\
&=(16z^8 +32z^7 +24z^6 +8z^5+2z^4 + 4z^3 +6z^2 +4z+ 1)(2z^2 + 1)^2. &(\delta_{8,1}, \delta_{2,2}^2)
\end{align*}
When the exponent of $\xi_1$ is a multiple of 4, we may combine the conjugate pairs together to obtain (using Granger's notation)
\begin{align}
&~~~~N_n(\xi_1^{4s_1}\xi_2^{s_2})-2^{n-4}\nonumber\\
&=-\frac{1}{16}\left(1+2(-1)^{s_1+s_2}+(-1)^{s_2}\right)\rho_n(\delta_{2,1})
-\frac{1}{16}\left(2(-1)^{s_1}+(-1)^{s_2}\right)\rho_n(\delta_{2,2})\nonumber\\
&~~~-\frac{1}{16}(-1)^{s_2}\rho_n(\delta_{4,1})
-\frac{1}{16}(-1)^{s_1+s_2}\rho_n(\delta_{8,2})
-\frac{1}{16}(-1)^{s_1}\rho_n(\delta_{8,1}).
\end{align}
Setting $(s_1,s_2)=(0,0),(0,1),(1,0),(1,1)$, we obtain
 \cite[Theorem~8]{Gra19}.

\end{example}

\medskip

\begin{example}
Consider $q=3, \ell=3$. It is easy to check that $\xi_1=\langle x+1\rangle$  has order 9 and $\xi_2=\langle x^2+1\rangle$ has order 3. The group $\Ecal$ is isomorphic to $C_9\times C_3$ with generators $\xi_1$ and $\xi_2$. We  have
\begin{align*}
\Ecal_1&=\{\langle 1\rangle,\xi_1,\xi_1^8\xi_2^2\},\\
\Ecal_2&=\Ecal_1 \cup \{\xi_1^2,\xi_2,\xi_2^2,\xi_1^4\xi_2^2,\xi_1^5\xi_2,\xi_1^7\xi_2\}.
\end{align*}
Using $\om_3=\om_9^3$, we obtain
\begin{align*}
c_{1,(j_1,j_2)}&=1+\om_9^{j_1}+\om_9^{8j_1+6j_2},   \\
c_{2,(j_1,j_2)}&=c_{1,(j_1,j_2)}
+\om_9^{2j_1}+\om_9^{3j_2}+\om_9^{6j_2}+\om_9^{4j_1+6j_2}
+\om_9^{5j_1+3j_2}+\om_9^{7j_1+3j_2},\\
P_{j_1,j_2}(z)&=1+c_{1,(j_1,j_2)}z+c_{2,(j_1,j_2)}z^2.
\end{align*}
We have $P_{j_1,j_2}(z)={\bar P}_{9-j_1,3-j_2}(z)$ and
\begin{align*}
P_{1,1}(z)&=P_{5,1}(z)=1+(1+\om_9+\om_9^5)z+3\om_9z^2,\\
P_{2,1}(z)&=P_{4,1}(z)=1+(1+\om_9^2+\om_9^4)z+3\om_9^4z^2,\\
P_{3,1}(z)&=1+i\sqrt{3}z,\\
P_{6,1}(z)&=P_{9,1}(z)=1+\frac{3-i\sqrt{3}}{2}z,\\
P_{7,1}(z)&=P_{8,1}(z)=1+(1+\om_9^7+\om_9^8)z+3\om_9^7z^2,\\
P_{1,3}(z)&=1+(1+\om_9+\om_9^8)z+3z^2,\\
P_{2,3}(z)&=1+(1+\om_9^2+\om_9^7)z+3z^2,\\
P_{3,3}(z)&=1,\\
P_{4,3}(z)&=1+(1+\om_9^4+\om_9^5)z+3z^2.
\end{align*}
Using Maple, we obtain
\begin{align*}
&~~~\prod_{(j_1,j_2)\ne(9,3)}P_{j_1,j_2}(z)&\\
&=(1+3z^2)\left(1+3z+3z^2\right)^2&(\eps_{2,2},  \eps^2_{2,3})\\
&~\times\left(1+3z+9z^2+15z^3+27z^4+27z^5+27z^6\right)^2&(\eps^2_{6})\\
&~\times \Big(1+6z+18z^2+39z^3+63z^4+81z^5+117z^6+243z^7& \\
&~~~+567z^8+1053z^9+1458z^{10}+1458z^{11}+729z^{12}\Big)^2. &(\eps^2_{12,4})
\end{align*}
It follows from Corollary~\ref{cor:trivial} that
\begin{align}\label{eq:Q3L3}
N_n(0,0,0)&=2^{n-3}-\frac{1}{27}\Big(\rho_n(\eps_{2,2})+2\rho_n(\eps_{2,3})
 +2\rho_n(\eps_{6}) +2\rho_n(\eps_{12,4})\Big).
\end{align}

We note that this formula is much simpler than the one given in \cite[Theoem~14]{Gra19}.

\end{example}

Finally we include an example with three generators.
\begin{example}
Consider $q=2, \ell=5$. By Lemma~\ref{lemma:generator}, the generators are $\xi_1=\langle x+1\rangle$, $\xi_2=\langle x^3+1\rangle$ and $\xi_3=\langle x^5+1\rangle$, of orders 8, 2, and 2, respectively.
We have
\begin{align}
\Ecal_1&=\{\langle 1\rangle,\xi_1\}, \nonumber \\
\Ecal_2&=\Ecal_1\cup \{ \xi_1^2,\xi_1^7\xi_2\},\nonumber \\
\Ecal_3&=\Ecal_2 \cup\{\xi_2,\xi_1^2\xi_2\xi_3,\xi_1^3,\xi_1^5\xi_2\xi_3\},\nonumber \\
\Ecal_4&=\Ecal_3\cup\{\xi_1\xi_2,\xi_1^3\xi_2\xi_3,\xi_1^4,\xi_1^4\xi_2,\xi_1^5\xi_3,\xi_1^6,\xi_1^6\xi_2\xi_3,
\xi_1^7\xi_3\}.
\end{align}
Using \eqref{eq:cdB},  we have
\begin{align*}
c_{1, (j_1,j_2,j_3)}&=1+\om_8^{j_1},\\
c_{2, (j_1,j_2,j_3)}&=c_{1, (j_1,j_2,j_3)}+\om_8^{2j_1}+\om_8^{7j_1}(-1)^{j_2},\\
c_{3, (j_1,j_2)}&=c_{2, (j_1,j_2,j_3)}+\om_8^{2j_1}(-1)^{j_2+j_3}+\om_8^{3j_1}+\om_8^{5j_1}(-1)^{j_2+j_3}+(-1)^{j_2},\\
c_{4, (j_1,j_2)}&=c_{3, (j_1,j_2,j_3)}+\om_8^{j_1}(-1)^{j_2}
+\om_8^{3j_1}(-1)^{j_2+j_3}+\om_8^{4j_1}+\om_8^{4j_1}(-1)^{j_2}\\
&~~~+\om_8^{5j_1}(-1)^{j_3}+\om_8^{6j_1}+\om_8^{6j_1}(-1)^{j_2+j_3}+\om_8^{7j_1}(-1)^{j_3},\\
P_{j_1,j_2,j_3}(z)&=1+c_{1, (j_1,j_2,j_3)}z+c_{2,(j_1,j_2,j_3)}z^2+c_{3,(j_1,j_2,j_3)}z^3++c_{4,(j_1,j_2,j_3)}z^4.
\end{align*}
Using Maple, we obtain
\begin{align*}
P_{8,0,1}(z)&=(1+2z^2)(1+2z+2z^2),&(\delta_{2,2},\delta_{2,1})\\
P_{8,1,0}(z)&=1+2z+2z^2,&(\delta_{2,1})\\
P_{8,1,1}(z)&=1+2z+2z^2+4z^3+4z^4, &(\delta_{4,1})\\
P_{4,0,0}(z)&=1,&\\
P_{4,0,1}(z)&=(1+2z+2z^2)(1-2z+2z^2),&(\delta_{2,1},\delta_{2,3})\\
P_{4,1,0}(z)&=1+2z^2,&(\delta_{2,2})\\
P_{4,1,1}(z)&=1+2z^2+4z^4,&(\delta_{4,2})\\
P_{2,0,0}(z)P_{6,0,0}(z)&=1+2z+2z^2,&(\delta_{2,1})\\
P_{2,0,1}(z)P_{6,0,1}(z)&=(1+2z+2z^2)(1-2z+2z^2)(1+2z+2z^2+4z^3+4z^4),&(\delta_{4,1},\delta_{2,1},\delta_{2,3})\\
P_{2,1,0}(z)P_{6,1,0}(z)&=1+2z+2z^2+4z^3+4z^4,&(\delta_{4,1})\\
P_{2,1,1}(z)P_{6,1,1}(z)&=16z^8 + 16z^7 + 8z^6 - 4z^4 + 2z^2 + 2z + 1,&(\delta_{8,3})
\end{align*}
and
\begin{align*}
&~~~P_{1,0,0}(z)P_{7,0,0}(z)P_{3,0,0}(z)P_{5,0,0}(z) &\\
&=(16z^8 +32z^7 +24z^6 +8z^5+2z^4 + 4z^3 +6z^2 +4z+ 1)(2z^2 + 1)^2, &(\delta_{8,1},\delta_{2,2}^2)\\
&~~~P_{1,0,1}(z)P_{7,0,1}(z)P_{3,0,1}(z)P_{5,0,1}(z) &\\
&=(16z^8 +8z^6-8z^5+2z^4 -4z^3 +2z^2 + 1)(1+2z+2z^2+4z^3+4z^4)^2,&(\delta_{8,4},\delta_{4,1}^2)\\
&~~~P_{1,1,0}(z)P_{7,1,0}(z)P_{3,1,0}(z)P_{5,1,0}(z)&\\
&=(16z^8 + 8z^6 + 8z^5 + 2z^4 + 4z^3 + 2z^2 + 1)(1+2z+2z^2)^2,&(\delta_{8,2},\delta_{2,1}^2),\\
&~~~P_{1,1,1}(z)P_{7,1,1}(z)P_{3,1,1}(z)P_{5,1,1}(z) &\\
&=(16z^8 +32z^7 +24z^6 +8z^5+2z^4 + 4z^3 +6z^2 +4z+ 1)(1+2z^2 +4z^4)^2. &(\delta_{8,1},\delta_{4,2}^2)
\end{align*}
By combining conjugate pairs as before, we obtain (using the notation from \cite{Gra19})\\
\begin{align}
&~~~~N_n(\xi_1^{4s_1}\xi_2^{s_2}\xi_3^{s_3})-2^{n-5}\nonumber\\
&=-\frac{1}{32}(-1)^{s_3}\Big(\rho_n(\delta_{2,2})+\rho_n(\delta_{2,1})\Big)- \frac{1}{32}(-1)^{s_2}\rho_n(\delta_{2,1})-\frac{1}{32}(-1)^{s_2+s_3}\rho_n(\delta_{4,1})\nonumber\\
&~~~-\frac{1}{32}(-1)^{s_3}\Big(\rho_n(\delta_{2,3})+\rho_n(\delta_{2,1})\Big)- \frac{1}{32}(-1)^{s_2}\rho_n(\delta_{2,2})- \frac{1}{32}(-1)^{s_2+s_3}\rho_n(\delta_{4,2})\nonumber\\
&~~~- \frac{1}{32}\rho_n(\delta_{2,1})
-\frac{1}{32}(-1)^{s_3}\Big(\rho_n(\delta_{4,1})+\rho(\delta_{2,3})+\rho_n(\delta_{2,1})\Big)
- \frac{1}{32}(-1)^{s_2}\rho_n(\delta_{4,1})\nonumber\\
&~~~- \frac{1}{32}(-1)^{s_2+s_3}\rho_n(\delta_{8,3})
-\frac{1}{32}(-1)^{s_1}\Big(\rho_n(\delta_{8,1})+2\rho_n(\delta_{2,2})\Big)\nonumber\\
&~~~-\frac{1}{32}(-1)^{s_1+s_3}\Big(\rho_n(\delta_{8,4})+2\rho_n(\delta_{4,1})\Big)
-\frac{1}{32}(-1)^{s_1+s_2}\Big(\rho_n(\delta_{8,2})+2\rho_n(\delta_{2,1})\Big)\nonumber\\
&~~~-\frac{1}{32}(-1)^{s_1+s_2+s_3}\Big(\rho_n(\delta_{8,1})+2\rho_n(\delta_{4,2})\Big).
\end{align}
Thus
\begin{align*}
&~~~~N_n(\xi_1^{4s_1}\xi_2^{s_2}\xi_3^{s_3})-2^{n-5}\\
&=-\frac{1}{32}\left(1+(-1)^{s_3}+(-1)^{s_2}+2(-1)^{s_3}+2(-1)^{s_1+s_2}\right)\rho_n(\delta_{2,1})\\
&~~~-\frac{1}{32}\left((-1)^{s_3}+(-1)^{s_2}+2(-1)^{s_1}\right)\rho_n(\delta_{2,2})
-\frac{2}{32}(-1)^{s_3}\rho_n(\delta_{2,3})\\
&~~~-\frac{1}{32}\left((-1)^{s_2+s_3}+(-1)^{s_3}+(-1)^{s_2}+2(-1)^{s_1+s_3}\right)\rho_n(\delta_{4,1})\\
&~~~-\frac{1}{32}\left((-1)^{s_2+s_3}+2(-1)^{s_1+s_2+s_3}\right)\rho_n(\delta_{4,2})\\
&~~~-\frac{1}{32}\left((-1)^{s_1}+ (-1)^{s_1+s_2+s_3}\right)\rho_n(\delta_{8,1})
-\frac{1}{32}(-1)^{s_1+s_2}\rho_n(\delta_{8,2})\\
&~~~-\frac{1}{32}(-1)^{s_2+s_3}\rho_n(\delta_{8,3})
-\frac{1}{32}(-1)^{s_1+s_3}\rho_n(\delta_{8,4}).
\end{align*}

This immediately implies \cite[Theorem~12]{Gra19} by taking $(s_1,s_2,s_3)=(0,0,0)$, $(1,0,0)$,$(0,0,1)$,$(1,0,1)$, respectively.
\end{example}

\section{Conclusion} \label{conclusion}

Through the study of the group of equivalent classes of monic irreducible polynomials with prescribed coefficients,  we obtain general expressions for the generating functions  of the number of monic irreducible polynomials with prescribed coefficients over finite fields.  Explicit formulae can be obtained accordingly. We demonstrate our recipe using several concrete examples  and compare our expressions with previous known results.

\end{document}